\documentclass[10pt]{article}
\usepackage[latin1]{inputenc}
 \usepackage[dvips]{graphicx}
 \usepackage[dvipsnames]{xcolor}
 \usepackage{amsmath}

 \usepackage{amsfonts}
 \usepackage{amssymb}
 \usepackage{layout}
 \usepackage{verbatim}
 \usepackage{todonotes}
 \usepackage{alltt}
 \usepackage{mathrsfs}
 \usepackage{cite}
\usepackage{rotating}
\usepackage{longtable}
\usepackage{tabu}
\usepackage{array}
\usepackage{tikz-cd}
\usepackage{titlesec}
\usepackage{setspace}
\usepackage[standard,amsmath,thref,hyperref, thmmarks]{ntheorem}
\usepackage{mathtools}
\usepackage{lipsum}

\usepackage{hyperref}
\hypersetup{
    linktocpage=true,
    colorlinks=true,
    linkcolor=Red,
    filecolor=orange,      
    urlcolor=cyan,
    citecolor=teal,
}

\titleformat*{\section}{\normalsize\bfseries}

 \input xy
 \xyoption{all}
 
\oddsidemargin - 1.5 pt
\newtheorem{thm}{Theorem}
\newtheorem{lem}[thm]{Lemma}
\newtheorem{prop}[thm]{Proposition}
\newtheorem{cor}[thm]{Corollary}
\newtheorem{dfn}[thm]{Definition}
\newtheorem{conj}[thm]{Conjecture}
\newtheorem{ques}[thm]{Question}

\newtheorem{rem}[thm]{Remark}

\renewtheorem{curve}{Curve}

\theoremstyle{nonumberplain}
 \theoremheaderfont{\itshape}
 \theorembodyfont{\normalfont}
  \theoremseparator{.}
\renewtheorem{proof}{Proof}

\newcommand{\C}{\mathbb C}

\newcommand{\R}{\mathbb R}
\newcommand{\Q}{\mathbb Q}
\newcommand{\Z}{\mathbb Z}
\newcommand{\F}{\mathbb F}

\def\l{{\ell}}
\def\O{{\cal O}}

\def\m{{\mathfrak m}}
\def\n{{\mathfrak n}}
\def\p{{\mathfrak p}}

\def\GK{{\text{Gal}(\overline{K} / K)}}

\title{Examples of genuine QM abelian surfaces which are modular}
\date{}
\author{Ciaran Schembri}

\begin{document}
\bibliographystyle{plain}

\maketitle

\begin{abstract}Let $K$ be an imaginary quadratic field. Modular forms for GL(2) over $K$ are known as Bianchi modular forms. Standard modularity conjectures assert that every weight 2 rational Bianchi newform has either an associated elliptic curve over $K$ or an associated abelian surface with quaternionic multiplication over $K$. We give explicit evidence in the way of examples to support this conjecture in the latter case. Furthermore, the quaternionic surfaces given correspond to \textit{genuine} Bianchi newforms, which answers a question posed by J. Cremona as to whether this phenomenon can happen.
\end{abstract}

\section{Introduction}

Let $K$ be an imaginary quadratic field. A simple abelian surface over $K$ whose algebra of $K$-endomorphisms is an indefinite quaternion algebra over $\Q$ is commonly known as a \textit{QM abelian surface}, or just \textit{QM surface}. These are also often referred to \textit{false elliptic curves}, coined by J.-P. Serre  in the 1970s \cite{Adimoolam77} based on the observation that such a surface is isogenous to the square of an elliptic curve modulo every prime of good reduction \cite[Lemma 6]{Yoshida73}. 

It is well-known that one can obtain QM surfaces over $K$ by base changing suitable abelian surfaces over $\Q$. Accordingly, let us call a QM surface over $K$ \textit{genuine} if it is not the twist of base-change to $K$ of an abelian surface over $\Q$. Motivated by the conjectural connections with Bianchi modular forms, in 1992 J. Cremona asked whether genuine QM surfaces over imaginary quadratic fields should exist (see Question \ref{ques:Cremona}). We answer this question in the positive by providing explicit genus 2 curves whose Jacobians are genuine QM surfaces. To the best of our knowledge these are the first such examples in the literature. Furthermore, by carrying out a detailed analysis of the associated Galois representations and applying the Faltings-Serre-Livn\'{e} criterion, we prove the modularity of these QM surfaces. The main result of the present article is as follows:

\begin{thm}\label{examples_thm} The Jacobians of the following genus 2 curves are QM surfaces which are modular by a genuine Bianchi newform as in Conjecture \ref{modularity}. 
	\begin{enumerate}
		\item $C_1: y^2= x^6 + 4ix^5 + (-2i - 6)x^4 + (-i + 7)x^3 + (8i - 9)x^2 - 10ix + 4i + 3, \\  \text{Bianchi newform: }$  \href{http://www.lmfdb.org/ModularForm/GL2/ImaginaryQuadratic/2.0.4.1/34225.3/a/}{2.0.4.1-34225.3-a;} 
		\item $C_2 : y^2=x^6 + (-2\sqrt{-3} - 10)x^5 + (10\sqrt{-3} + 30)x^4 + (-8\sqrt{-3} - 32)x^3 
		\\+ (-4\sqrt{-3} + 16)x^2 +     (-16\sqrt{-3} - 12)x - 4\sqrt{-3} + 16, \\  \text{Bianchi newform: }$ \href{http://lmfdb.org/ModularForm/GL2/ImaginaryQuadratic/2.0.3.1/61009.1/a/}{2.0.3.1-61009.1-a;}
		\item $C_3: y^2=(104\sqrt{-3} - 75)x^6 + (528\sqrt{-3} + 456)x^4 + (500\sqrt{-3} + 1044)x^3\\ + (-1038\sqrt{-3} + 2706)x^2 + (-1158\sqrt{-3} + 342)x - 612\sqrt{-3} - 1800, \\  \text{Bianchi newform: }$  \href{http://www.lmfdb.org/ModularForm/GL2/ImaginaryQuadratic/2.0.3.1/67081.3/a/}{2.0.3.1-67081.3-a;} 
		\item $C_4 : y^2 = x^6 - 2\sqrt{-3}x^5 + (2\sqrt{-3} - 3)x^4 + 1/3(-2\sqrt{-3} + 54)x^3\\ + (-20\sqrt{-3} + 3)x^2 + (-8\sqrt{-3} - 30)x + 4\sqrt{-3} - 11, \\  \text{Bianchi newform: }$ \href{http://www.lmfdb.org/ModularForm/GL2/ImaginaryQuadratic/2.0.3.1/123201.1/b/}{2.0.3.1-123201.1-b.}
	\end{enumerate}
\end{thm}

\begin{proof} See \S \ref{Faltings Serre}. \qed
\end{proof}

Over the rationals there is a celebrated result that establishes a connection between elliptic curves over $\Q$ and classical newforms of weight 2.  Extending this to number fields is an important aspect of the Langlands programme. However, in the case when the number field is totally complex the correspondence needs to be modified to include QM surfaces. This was first observed by P. Deligne in a letter to J. Mennicke in 1979 \cite{Deligne79Letter} for imaginary quadratic fields. The construction is detailed in \cite{EGM80} and is illustrated with an explicit example.

Thus the details of modularity for QM surfaces are different from the case of GL$_2$-type. Let us be completely explicit with the following conjecture \cite{Clo90,Taylor95}.

\begin{conj}\label{modularity} Let $K$ be an imaginary quadratic field.
	\begin{enumerate} 
		\item Let $f$ be a Bianchi newform over $K$ of weight $2$ and  level $\Gamma_0(\mathfrak{n})$ with rational Hecke eigenvalues. Then there is either an elliptic curve $E/K$ without CM by $K$ of conductor $\mathfrak{n}$ such that $L(E/K,s)=L(f,s)$ or there is a QM surface $A/K$ of conductor $\mathfrak{n}^2$ such that $L(A/K,s)= L(f,s)^2$. 
		
		\item Conversely, if $E/K$ is an elliptic curve without CM by $K$ of conductor $\mathfrak{n}$ then there is an $f$ as above such that $L(E,s)=L(f,s)$. Moreover, if $A/K$ is a QM surface of conductor $\mathfrak{n}^2$ then there is an $f$ as above such that $L(A,s)= L(f,s)^2$.
	\end{enumerate}
\end{conj}

$$\begin{tikzcd}[column sep=1cm]
& \Bigg\{ \substack{\text{weight 2 rational} \\ \text{Bianchi newforms}/K}   \Bigg\}  \arrow[dashed,leftrightarrow]{r}{1:1}
&  \substack{ \Bigg\{ \substack{\text{non-CM by $K$} \\ \text{elliptic curves}/K \\ \text{up to isogeny}}   \Bigg\}  \\ \text{\huge{$\sqcup$}} \\ \Bigg\{ \substack{\text{QM surfaces}/K \\ \text{up to isogeny}}   \Bigg\}  } 
\end{tikzcd}$$

Let $f$ be a classical newform of weight 2 with a real quadratic Hecke eigenvalue field $K_f= \Q(\{ a_i\})$ and denote $<\sigma>=\text{Gal}(K_f / \Q)$. We say that $f$  has an \emph{inner twist} if $f^\sigma = f \otimes \chi_K$ where $\chi_K$ is the quadratic Dirichlet character associated to some imaginary quadratic field $K$. It follows that $f$ and $f^\sigma$ must base change to the same Bianchi newform over $K$. The term \emph{genuine} is used for newforms that are not (the twist of) base-change of a classical newform. For more background on Bianchi newforms see \cite[\S 2]{CW94}.

From a geometric point of view, let $A/\Q$ be the abelian surface of GL$_2$-type corresponding to the newforms $f, f^\sigma$ with an inner twist  via $L(A/\Q,s)=L(f,s)L(f^\sigma,s)$.  If the base-change surface $A \otimes_{\Q} K$ remains simple then it is necessarily a QM surface and $L(A/K,s) = L(F,s)^2$, where $F$ is the induction from $\Q$ to $K$ of $f$. This motivates the following question (see \cite[Question 1']{Cremona92} and also \cite[Conjecture 1]{DGP10}).

\begin{ques}\label{ques:Cremona}
	If $f$ is a rational weight $2$ Bianchi newform over $K$ which is \textbf{genuine}, does $f$ have an associated elliptic curve over $K$?
\end{ques}

Given the modular correspondence above, we could rephrase this question to ask whether all QM surfaces arise from a GL$_2$-type surface over $\Q$. The genus 2 curves given in Theorem \ref{examples_thm} answer this question and say that such a newform $f$ does not necessarily have to correspond to an elliptic curve. 

The genuine QM surfaces we present also have an interesting connection to the Paramodularity Conjecture. Recall that the Paramodularity Conjecture posits a correspondence between abelian surfaces $A/\Q$ with $\text{End}_\Q(A) \otimes \Q \simeq \Q$ and genus 2 paramodular rational Siegel newforms of weight 2 that are not Gritsenko lifts \cite[Conjecture 1.1]{BK14modified}. It has been recently pointed out by F. Calegari et al. \cite[\S10]{BCGP18} that the conjectural correspondence needs to include abelian 4-folds $B/\Q$ with $\text{End}_\Q(B) \otimes \Q$ an indefinite non-split quaternion algebra over $\Q$ (see the amended version in \cite[\S8]{BK14modified}) . This can be illustrated using our genuine QM surfaces.

Let $C/K$ be any of the four curves given in Theorem \ref{examples_thm}. Define $A/K$ to be the QM surface given by taking the Jacobian of $C/K$ with \(\text{End}_K(A) \otimes \Q \simeq D/\Q\) an indefinite non-split quaternion algebra. Then the Weil restriction $B=\text{Res}_{K/\Q}(A)$ of $A$ from $K$ to $\Q$ is a simple abelian 4-fold such that $\text{End}_\Q(B) \otimes \Q \simeq D/\Q$. We prove that there is a genuine rational weight 2 Bianchi newform $f$ over $K$ such that $L(A/K,s)=L(f,s)^2$. Now let $F$ be the genus 2 paramodular rational Siegel newform of weight 2 that is the theta lift of $f$. It now follows from the properties of Weil restriction \cite{Milne72} and theta lifting \cite{BDPS15} that $L(B/\Q,s)=L(A/K,s)=L(f,s)^2=L(F,s)^2$.

In analogy to the case of QM surfaces, at any prime $p$ unramified in $D$ the 8-dimensional $p$-adic Tate module of $B/\Q$ splits as the square of a 4-dimensional submodule\cite[\S7]{Chi92}. Then the 4-dimensional $p$-adic Galois representation has similar arithmetic to one that arises from an abelian surface over $\Q$ with trivial endomorphisms. Indeed, our example above shows that via the representation afforded by the submodule, $B/\Q$ corresponds to a Siegel newform of the type considered in the Paramodularity Conjecture.

The article will be laid out as following: in \S\ref{surfaces section} we outline how these genus 2 curves were found and in \S \ref{rep section} we discuss some arithmetic properties of the attached Galois representation in the case where $\ell$ divides the discriminant of the acting quaternion algebra. Then \S\ref{Faltings Serre} will be dedicated to showing how the Faltings-Serre-Livn\'{e} criterion can be applied in order to prove that the examples given are modular. The final section lists the examples and contains further details of interest about them.\\

\section{Rational points on Shimura curves}\label{surfaces section}


In this section we outline how the genus 2 curves in Theorem \ref{examples_thm} were found. Let $A$ be a geometrically simple abelian surface defined over an imaginary quadratic field $K$.
Define the endomorphism rings End$_K(A)$ and End$_{\overline{K}}(A)$ to be the endomorphisms of $A$ which are defined over $K$ and $\overline{K}$ respectively. As in the introduction, we use the convention that $A$ has \textit{quaternionic multiplication}, or \textit{QM} for short, if End$_{{K}}(A)$ is an order $\O$ in an indefinite quaternion algebra $B_D$ over $\Q$. We say that $A$ has \textit{potential QM} if the action of $\O$ is defined over some extension of $K$. The notation $B_D$ is used for the unique quaternion algebra of discriminant $D$ up to isomorphism. Note that $B_D$ must be non-split because $A$ is simple. For the remainder of the article let $\O$ be a maximal order of $B_D$.

Families of QM surfaces have been constructed by K. Hashimoto et al. (see \cite{HT99}) for quaternion algebras of discriminant 6 and 10. Testing numerically  it would seem these give rise to surfaces which are all (a twist of) base-change. So we instead utilise two families given by S. Baba and H. Granath \cite{BG08} that have been derived from the moduli space.

Given the order $\O$, the set of norm 1 elements is denoted by $\O^1$. These act as isometries on the upper half plane $\mathcal{H}_2$ via an embedding $\O \hookrightarrow M_2(\R)$ and the resulting quotient $X_D = \mathcal{H}_2 / \O^1$ is a moduli space for abelian surfaces with quaternionic multiplication by $\O$ \cite{BG08}. It is well known that these are compact Riemann surfaces called Shimura curves and they admit a model defined over $\Q$. In particular, these are $X_6: X^2+ 3Y^2 + Z^2 = 0$ and $X_{10}: X^2+ 2Y^2 + Z^2 = 0$.

Let us detail the family for $D=6$, for more information see \cite{BG08}. We take the point on the conic 
$$P_j=(4: 3\sqrt{j} : \sqrt{-27j-16}) \in X_6$$
and define the genus 2 curve
\begin{align*}
C_j: y^2 &= (-4+3s)x^6 + 6tx^5 + 3t(28+9s)x^4 - 4t^2x^3\\ 
&+ 3t^2(28-9s)x^2 + 6t^3x - t^3(4+3s),
 \end{align*}
where $t=-2(27j+16)$ and $s=\sqrt{-6j}$. Then the Jacobian of $C_j$ is a QM surface. The curve is defined over the field $\Q(\sqrt{j},\sqrt{-6})$ and the field of moduli for $C_j$ is $\Q(j)$. In this way we can generate numerous QM surfaces by, for example, taking any $j \in \Q$.

For the purposes of modularity we need to fix a field $K=\Q(\sqrt{-\delta})$ and define a QM surface over $K$. To do this we first establish whether $X_6(K)$ is non-empty and if this is the case then find a $K$-rational point $(a:b:c)$ on $X_6$. Take the quantity $j=(\frac{4b}{3a})^2 \in K$ and the corresponding genus 2 curve $C_j$ is defined over $K(\sqrt{-6})$. It has a model defined over $K$ if and only if $K$ splits the Mestre obstruction which is the quaternion algebra
$$\Bigl( \frac{-6j, -2(27j+16)}{\Q(j)} \Bigr) \simeq \Bigl( \frac{-6, 2}{K} \Bigr) \simeq B_D \otimes_\Q K  .$$
Hence $C_j$ has a model defined over $K$ if and only if $K$ splits $B_D$. Since $D=6$ this happens exactly when neither of 2 or 3 split in $K$.

So let us suppose that $K \hookrightarrow B_D$. Once we have the curve $C_j$ defined over $K(\sqrt{-6})$ we wish to find an isomorphic curve defined over $K$. Using MAGMA it is possible to take Igusa-Clebsch invariants and then create a model defined over $K$ with the same Igusa-Clebsch invariants. This then allows us to test whether the curve is a twist of base-change. It can be easily shown that a curve is genuine if the Euler polynomials at a pair of conjugate primes are not the same (up to twists).  If it is indeed genuine then we endeavour to find a smaller model for the curve.

The size of the level places a limitation on whether a Bianchi modular form can be computed. Hence we try to find surfaces with as small a conductor as possible. With QM surfaces it can be a challenge to find examples with small conductor (see \cite[Section 8]{BK14modified}).

It is necessary to know the conductor exactly since we wish to find the conjecturally associated Bianchi newform.  The odd part of the conductor can be found using MAGMA. Computing the even part has recently been made possible using machinery developed in \cite{DD17}. The support of the ideal generated by the discriminant of a genus 2 hyperelliptic curve contains the support of the conductor of its Jacobian and the inclusion can in fact be strict. This phenomenon arises especially when one works with curves that have very large coefficients. 

The curves in Theorem \ref{examples_thm} were then found by parameterising the conic $X_6(K)$ and conducting a large search with varying $j$-value. To control the support of the conductor and stay away from large primes we used the proposition below. Once suitable curves were discovered, minimal models were found using as yet unpublished code by L. Demb\'{e}l\'{e}.

\begin{prop}
	Let $C_j/k$ be a genus 2 curve as above. Then $C_j$ has potentially good reduction at a prime $\p \nmid 6$ if and only if $\nu_\p(j)=0$.
\end{prop}

\begin{proof}
See \cite[Proposition 3.19]{BG08}.
\end{proof}

\section{Galois representations attached to QM surfaces}\label{rep section}


In this section we describe the image of the Galois representation attached to a QM surface when the prime $\ell$ divides the discriminant of the quaternion algebra. For a brief overview on the arithmetic of quaternion algebras see \cite[Ch. 2]{MR03}.

Let $A/K$ be a QM surface with $\O \hookrightarrow \text{End}_K(A)$ a maximal order in the quaternion algebra $B/\Q$ and denote by $\sigma_{\ell} :  G_K \rightarrow GL_4 (\Z_\ell)$ the representation coming from the $\ell$-adic Tate module $T_\ell A= \displaystyle\lim_{\leftarrow n} \ A[\ell^n]$. 

Denote by $\O_\ell = \O \otimes_\Z \Z_\ell$ and $B_\ell = B \otimes_\Q \Q_\ell$. Then for each prime $\ell$ the Tate module $T_\l A$ is free of rank 1 over $\O_\l$ \cite{Ohta74} and hence there is an associated $\ell$-adic representation 
$$\rho_{\ell} : \ G_K \longrightarrow \text{Aut}_{\O}(T_\ell A) \simeq \O_\ell^\times \subseteq B_\ell^\times.$$
Furthermore, the $\rho_{\ell} $ form a strictly compatible system of $\ell$-adic representations \cite[\S 5]{Jordan86}.

If $\ell \nmid \text{Disc}(B)$  this precisely means that $(\O \otimes \Z_\ell)^{\times} \simeq GL_2(\Z_\ell)$ and in this case there is a decomposition \cite[Theorem 7.1]{Chi92}
$$\sigma_{\ell} \simeq  \rho_{\ell} \oplus \rho_{\ell}.$$

For the remainder of the section let $\ell$ be a prime that divides $\text{Disc}(B)$. This means that $\ell$ is ramified in $B$ and so $B_\ell$ is isomorphic to the unique division quaternion algebra over $\Q_\ell$. It can be represented as
$$\Big( \frac{ \pi, u}{\Q_\ell} \Big) \simeq \Q_\ell \cdot 1 + \Q_\ell \cdot i + \Q_\ell \cdot j + \Q_\ell \cdot ij; \ i^2=u, \ j^2=\pi;$$
where $\pi$ is the uniformiser of $\Z_\ell$ and $\Q_\ell({\sqrt{u}})$ is the unique unramified quadratic extension of $\Q_\ell$. 

Any quadratic extension of $\Q_\ell$ splits the ramified quaternion algebra. So let us denote $L=\Q_\ell({\sqrt{u}})$ and $R_L$ as its ring of integers. Then $B \otimes_{\Q_\ell} L \simeq M_2(L)$ and there is an explicit isomorphism of $\Q_\ell$-algebras
\begin{align}B_\ell \simeq  \Bigl\{ \Bigl( \begin{matrix}  \alpha & \beta \\ \pi \beta' & \alpha' \end{matrix} \Bigr)  &| \  \alpha, \beta \in L, \ ':L\rightarrow L \text{ is conjugation in } L/\Q_\l  \Bigr\}  \subseteq M_2(L); \\
&i \mapsto  \Bigl( \begin{matrix}  \sqrt{u} & 0 \\ 0 & -\sqrt{u} \end{matrix} \Bigr),  \ 
\ j \mapsto \Bigl( \begin{matrix}  0 & 1 \\ \pi & 0 \end{matrix} \Bigr).
\end{align}

If $\nu: \Q_\l \rightarrow \Z$ is the $\l$-adic valuation then $w=\nu \circ N_{B_\l / \Q_\l}$ defines a valuation on the quaternion algebra $B_\l$. This gives us 
$$\O_\l = \{ \ x \in B_\l \ | \ w(x) \geq 0 \ \}$$
which is the unique maximal order of $B_\l$ and 
$$\mathcal{J} = \{ \ x \in B_\l \ | \ w(x) > 0 \ \},$$
a two-sided ideal. It is a principal ideal given by $\mathcal{J}= \O_\l j$ and any two-sided ideal of $\O_\l$ is a power of $\mathcal{J}$. In the isomorphism (1) we get $\O_\l$ by taking $\alpha, \beta \in R_L$.

Define $\lambda_\ell \subseteq \O$ to be the unique two-sided ideal of reduced norm $\ell$ such that $\lambda_\ell^2 =(\ell)$. The torsion subgroups $A[\ell]$ and $A[\lambda_\ell]$ are free of rank 1 over the $\F_\ell$-algebras $\O/\ell $ and $\O/\lambda_\ell$ respectively \cite[\S4]{Jordan86}. Explicitly, these have the structure
\begin{align*}
&\O/\ell \simeq  \bigl\{ \bigl( \begin{smallmatrix}  \alpha & \beta \\ 0 & \alpha^\l \end{smallmatrix} \bigr)  | \  \alpha, \beta \in \F_{\ell^2}  \bigr\}  \subseteq M_2(\F_{\ell^2}), \\
&\O/\lambda_\ell \simeq \F_{\ell^2}.
\end{align*}

Denote the residual representations by
\begin{align*}
\overline{\tau}_{\ell} :  \ &\GK \longrightarrow \text{Aut}_{\O}(A[\ell]) \leq \text{GL}_2( \F_{\ell^2}),  \\
\overline{\rho}_{\ell} : \  &\GK \longrightarrow \text{Aut}_{\O}(A[\lambda_\ell]) \simeq \F_{\ell^2}^\times.
\end{align*}
There is a commutative diagram:
$$\begin{tikzcd}
& G_K \arrow{rd}{\overline{\rho}_{\l}} \arrow{r}{\rho_\l} 
&  \arrow{d}  \text{Aut}_\O(T_\l A)  \\ 
&
& \text{Aut}_\O(A[\lambda_\l]).
\end{tikzcd}$$

Under the identification (1) and projecting as in the commutative diagram the image of $\overline{\rho}_{\ell}$ will lie in GL$_2(\F_{\l^2})$. Furthermore, up to conjugation it can be assumed that the image of $\overline{\rho}_{\ell}$ is contained in GL$_2(\F_\l)$ by \cite[Lemma 3.1]{Jones16}. Specifically, it will be contained in the non-split Cartan subgroup of GL$_2(\F_\l)$, which is the unique cyclic subgroup of order $\ell^2-1$.

For a Frobenius element $F_\upsilon$ the Hecke polynomial is given by 
$$P_{\rho_{\ell}}(F_\upsilon) = N_{B_\ell/\Q_\l} ( 1- \rho_{\ell}(F_\upsilon )t) = 1 - a_\upsilon t + N_\upsilon t^2,$$
which residually is 
$$ P_{\rho_{\ell}}(F_\upsilon) \text{ mod } \l = (1 - \alpha t)( 1- \alpha^\l t)$$
with $\alpha \in \F_{\l^2}^\times.$ Furthermore, the map $F_\upsilon \mapsto N_{\F_{\l^2} / \F_\l} (\alpha)$ is the cyclotomic character \cite{Jordan86}.

\begin{thm}\label{mod ses}
	Let $A$ be a geometrically simple abelian surface defined over an imaginary quadratic field $K$ such that $\text{End}_K(A)$ is isomorphic to a maximal order $\O$ in an indefinite quaternion algebra $B/\Q$. Suppose that the prime $\ell$ divides $\text{Disc}(B)$ and $\overline{\tau}_{\ell},\overline{\rho}_{\ell}$ are the residual Galois representations on the torsion subgroups  $A[\ell]$ and $A[\lambda_\ell]$ respectively. Then there is a short exact sequence of groups
	$$1 \longrightarrow \overline{\epsilon} \longrightarrow  \text{Im}(\overline{\tau}_{\ell}) \longrightarrow \text{Im}(\overline{\rho}_{\ell})  \longrightarrow 1,$$
	where $\overline{\epsilon} \leq \F_{\ell^2}^+.$
\end{thm}

	\begin{proof}
	Since  Aut$_\O (A[\l]) \simeq (\O/\l)^\times$ and Aut$_\O (A[\lambda_\l])  \simeq (\O/\lambda_\l)^\times \simeq \F_{\l^2}^\times$ it is enough to show that there is a short exact sequence
	$$1 \longrightarrow \F_{\l^2}^+ \longrightarrow  (\O/\l)^\times \longrightarrow  \F_{\l^2}^\times  \longrightarrow 1.$$
	
	Let $r$ be the projection $r:  (\O/\l)^\times \rightarrow (\O/\lambda_\l)^\times.$ Then $\text{ker}(r)$ consists of the cosets $\phi + \l$ such that $\phi \in \lambda_\l+1$. It follows that ${\rm ker}(r) \simeq (1+\lambda_\l)/(1+\l)$ which is isomorphic to $\F_{\l^2}^+.$  	  \qed
\end{proof}

\begin{rem}
	If $f$ is a Bianchi newform which corresponds to a QM surface with quaternion algebra $B_D$, then the residual representation attached to $f$ has cyclic image at the primes dividing the discriminant $D$.
\end{rem}

Given a Bianchi newform with rational coefficients, it would be desirable to have a criterion which determines whether $f$ should correspond to an elliptic curve or a QM surface. The above gives a necessary condition for $f$ to correspond to a QM surface. We wish to know whether a sufficient condition also exists and if so whether it can be determined from computing the trace of Frobenius for a finite set of primes.

\section{Proof of modularity}\label{Faltings Serre}

Here we provide a proof that the second example in Theorem \ref{examples_thm} is modular using the Faltings-Serre-Livn\'{e} criterion, the other cases follow similarly. So let $C_2$ be the genus 2 curve
\begin{align*}C_2 : y^2=&x^6 + (-2\sqrt{-3} - 10)x^5 + (10\sqrt{-3} + 30)x^4 + (-8\sqrt{-3} - 32)x^3 \\
&+ (-4\sqrt{-3} + 16)x^2 +     (-16\sqrt{-3} - 12)x - 4\sqrt{-3} + 16.
\end{align*} 
and $A$ the Jacobian of $C_2$.  The surface $A=Jac(C_2)$ has conductor $\p_{13,1}^4 \cdot \p_{19,1}^4$ with norm $61009^2$ and $\O \hookrightarrow \text{End}_{\Q(\sqrt{-3})}(A)$ where $\O$ is the maximal order of the rational quaternion algebra of discriminant 10. The endomorphism algebra can be independently verified using the machinery developed in \cite{CMSV18}.

Let $f \in S_2(\Gamma_0(\p_{13,1}^2 \cdot \p_{19,1}^2))$ be the genuine Bianchi newform which is listed on the LMFDB database with label
\href{http://lmfdb.org/ModularForm/GL2/ImaginaryQuadratic/2.0.3.1/61009.1/a/}{2.0.3.1-61009.1-a.} We will show that $f$ is modular to $A$. By the work of \cite{HST93,Taylor94} and more recently \cite{BH07, Mok14}, we can associate an $\ell$-adic Galois representation $\rho_{f,\ell}:{\rm Gal}(\overline{K}/K) \rightarrow {\rm GL}_2(\overline{\Q}_\l)$ to $f$ such that $L(f,s) = L(\rho_{f,\ell},s)$. 

The Faltings-Serre-Livn\'{e} method gives an effective way to prove that two Galois representations are isomorphic up to semisimplification by showing that the trace of Frobenius agree on a finite computable set of primes. We follow the steps outlined in \cite{DGP10} which for practical reasons necessitates use of the prime $\ell=2$. This prime is ramified in the acting quaternion algebra and as in \S\ref{rep section} we can associate a representation to the 2-adic Tate module.

\begin{lem} The representations
	$$\rho_{A,2}, \rho_{f,2} : \text{Gal}(\overline{\Q(\sqrt{-3})} / \Q(\sqrt{-3}) ) \longrightarrow \text{GL}_2(\overline{\Q}_2)$$
	have image contained in GL$_2(E)$, where $E$ is the unique unramified quadratic extension of $\Z_2$.
\end{lem}

\begin{proof}
	For $\rho_{A,2}$ this is a direct consequence of the way that the representation has been defined. Let us now consider $\rho_{f,2}$. The prime 31 is split in $\Q(\sqrt{-3})$ and the Hecke eigenvalues above these primes are distinct. Hence we can take the field by adjoining the roots of the Hecke polynomials which gives $\Q(\sqrt{-43},\sqrt{-123})$. The completion at either of the primes above 2 in this field gives the unique unramified quadratic extension of $\Z_2$ and so we can take this as the coefficient field $E$ by \cite[Corollary 1]{Taylor94}. \qed
\end{proof}

First we must show that the residual representations are isomorphic.

\begin{lem}
	The residual representations $\overline{\rho}_{A,2}, \overline{\rho}_{f,2}$ are isomorphic and have image $C_3 \subset \text{GL}_2( \F_2).$ 
\end{lem}

\begin{proof}
	Denote by $F_A$ and $F_f$ the fields cut out by  $\overline{\rho}_{A,2}$ and  $\overline{\rho}_{f,2}$ respectively. The field given by the 2-torsion of $A$ is isomorphic to $A_4$ which has only one proper normal subgroup. This subgroup has order 4 and so applying the short exact sequence of Theorem \ref{mod ses}, the image of $\overline{\rho}_{A,2}$ must be $C_3$. 
	
	We first note that it can be assumed Im$(\overline{\rho}_{f,2}) \subset GL_2(\F_2)$ due to the fact that the traces of Frobenius are all rational \cite[Lemma 3.1]{Jones16}. To show that Im$(\overline{\rho}_{f,2})=C_3$ let $\m$ denote the modulus
	$$\m = \p_2^3 \cdot \p_{13,1} \cdot \p_{19,1}.$$
	Then as explained in \cite[Ch. 6]{Jones15}, if Im$(\overline{\rho}_{f,2})$ is not equal to $C_3$ there must be a quadratic extension of $\Q(\sqrt{-3})$ contained in $F_f$ which corresponds to a quadratic character of Cl$(\Z[\frac{1+\sqrt{-3}}{2}], \m)$.
	We compute the ray class group to be
	$$\text{Cl}(\Z[\frac{1+\sqrt{-3}}{2}], \m) \simeq (\Z/2\Z)^2 \oplus (\Z/12\Z) \oplus (\Z/36\Z).$$
	
	Let us choose $\{\chi_1, \dots, \chi_4 \}$ as an $\F_2$-basis for the quadratic characters of Cl$(\Z[\frac{1+\sqrt{-3}}{2}], \m)$. Then  $\{\chi_1(\p), \dots, \chi_4(\p) \}_{\p \in S}$ spans $\F_2^4$, where $S=\{\p_{7,1}, \p_{7,2}, \p_{13,2}, \p_{19,2}, \p_5 \}$. If $F_f$ contains a quadratic subfield then by \cite[Proposition 5.4]{DGP10} the associated quadratic character must be non-zero for one of the primes in $S$. Hence there must be a prime $\p \in S$ that is inert in this subfield and so $\overline{\rho}_{f,2}(\text{Frob}_\p)$ must have order 2. However, we compute that the trace of Frobenius is odd for all primes in $S$ and therefore $F_f$ is a cubic extension of $\Q(\sqrt{-3})$.
	
	To show that the representations are isomorphic let $\psi_A$ denote the cubic character associated to $F_A$. Extend this to an $\F_3$-basis $\{\psi_A, \chi_1 \}$ of the cubic characters of Cl$(\Z[\frac{1+\sqrt{-3}}{2}], \m)$. We find that the prime $\p_{37,1}$ is such that $\psi_A(\p_{37,1}) = 0$ and $\chi_1(\p_{37,1}) \neq 0$. So if $\chi_f$ is the cubic character associated to $F_f$ and $\chi_f$ is not in the span of $\chi_A$ then  $\psi_f(\p_{37,1})$ must be non-zero. In particular, $\overline{\rho}_{f,2}(\text{Frob}_\p)$ must have order 3 but we find that Tr($\overline{\rho}_{f,2}(\text{Frob}_{\p_{37,1}})) = \text{Tr}(\overline{\rho}_{A,2}(\text{Frob}_{\p_{37,1}}))$ and so we can conclude that the residual representations are isomorphic. \qed
\end{proof}

Now that we have shown that the residual representations are isomorphic it remains to show that the full representations are isomorphic up to semisimplifcation. The residual images are cyclic and note that this will always be the case when the prime $\ell$ divides the discriminant of the acting quaternion algebra. Since the images are cyclic we can use Livn\'{e}'s criterion, which applies when the image is absolutely reducible.

\begin{thm}\label{thm:Livne}
	Let $K$ be a number field, $E$ a finite extension of $\Q_2$ and $\O_E$ its ring of integers with maximal ideal $\mathcal{P}$. Let
	$$\rho_1, \rho_2 : \text{Gal}(\overline{K}/K) \longrightarrow GL_2(E)$$
	be two continous representations unramified outside of a finite set of primes $S$ and $K_{2,S}$ the compositum of all quadratic extensions of $K$ unramified outside of $S$.
	
	Suppose that
	\begin{enumerate}
		\item $\text{Tr} (\rho_1) \equiv Tr(\rho_2) \equiv 0 \ (\text{mod } \mathcal{P})$ and $\text{Det}(\rho_1) \equiv \text{Det}(\rho_2) \equiv 1 \ (\text{mod } \mathcal{P})$;
		\item There is a finite set of primes $T$ such that the characteristic polynomials of $\rho_1$ and $\rho_2$ are equal on the set $\{ \text{Frob}_\p \ | \ \p \in T\}$. 
	\end{enumerate}
	Then $\rho_1$ and $\rho_2$ have isomorphic semisimplifications.
\end{thm}

\begin{proof}
	See \cite[Theorem 4.3]{Livne87}. \qed
\end{proof}

It is now possible to show that the representations attached to $A$ and $f$ are isomorphic up to semisimplification.

\begin{thm}\label{examples_thm} The Jacobians of the following genus 2 curves are QM surfaces which are modular by a genuine Bianchi newform as in Conjecture \ref{modularity}. 
	\begin{enumerate}
		\item $C_1: y^2= x^6 + 4ix^5 + (-2i - 6)x^4 + (-i + 7)x^3 + (8i - 9)x^2 - 10ix + 4i + 3, \\  \text{Bianchi newform: }$  \href{http://www.lmfdb.org/ModularForm/GL2/ImaginaryQuadratic/2.0.4.1/34225.3/a/}{2.0.4.1-34225.3-a;} 
		\item $C_2 : y^2=x^6 + (-2\sqrt{-3} - 10)x^5 + (10\sqrt{-3} + 30)x^4 + (-8\sqrt{-3} - 32)x^3 
		\\+ (-4\sqrt{-3} + 16)x^2 +     (-16\sqrt{-3} - 12)x - 4\sqrt{-3} + 16, \\  \text{Bianchi newform: }$ \href{http://lmfdb.org/ModularForm/GL2/ImaginaryQuadratic/2.0.3.1/61009.1/a/}{2.0.3.1-61009.1-a;}
		\item $C_3: y^2=(104\sqrt{-3} - 75)x^6 + (528\sqrt{-3} + 456)x^4 + (500\sqrt{-3} + 1044)x^3\\ + (-1038\sqrt{-3} + 2706)x^2 + (-1158\sqrt{-3} + 342)x - 612\sqrt{-3} - 1800, \\  \text{Bianchi newform: }$  \href{http://www.lmfdb.org/ModularForm/GL2/ImaginaryQuadratic/2.0.3.1/67081.3/a/}{2.0.3.1-67081.3-a;} 
		\item $C_4 : y^2 = x^6 - 2\sqrt{-3}x^5 + (2\sqrt{-3} - 3)x^4 + 1/3(-2\sqrt{-3} + 54)x^3\\ + (-20\sqrt{-3} + 3)x^2 + (-8\sqrt{-3} - 30)x + 4\sqrt{-3} - 11, \\  \text{Bianchi newform: }$ \href{http://www.lmfdb.org/ModularForm/GL2/ImaginaryQuadratic/2.0.3.1/123201.1/b/}{2.0.3.1-123201.1-b.}
	\end{enumerate}
\end{thm}

\begin{proof}
	Restricting the representations to the cubic extension cut out by the residual representation, the mod $\mathcal{P}$ image becomes trivial. We are then in a position to apply Livn\'{e}'s criterion.
	
	Let $\{ \chi_1, \dots, \chi_6\}$ be a basis of quadratic characters. Any set of primes  $\{ \p_i \}$ for which the vectors $\{( \chi_1(\p_i), \dots, \chi_6(\p_i)) \}$ cover $\F_2^6 \backslash \{0\}$ will satisfy the criterion. Following \cite[\S2.3 step (7)]{DGP10} we compute the set \(T = \{ 3, 37, 43, 61, 67, 73, 97, 103, 127, 151, 157, 193, 211, 307, 313, 343, 373, 433, 463, 499, 523,\\ 631, 661, 823, 1321, 2197, 2557, 2917  \}.\) The traces of Frobenius agree on this set.
	
	 To complete the proof we note that this shows that the representations are isomorphic when restricting to the cubic extension cut out by the residual representations. As explained in \cite[pp. 362]{SW05} this means that the full representations could differ by a character. We find that the prime above 5 is inert in the cubic extension and that the traces of Frobenius agree on this prime, which forces the character to be trivial. Hence we can conclude that the two representations are isomorphic up to semisimplification. 
	 
	The set of primes needed for the other three surfaces are:
	\begin{itemize}
	\item $T(C_1) = \{ 5, 17, 61, 73, 121, 125, 157  \}.$
	\item $T(C_3)= \{  3, 13, 19, 31, 43, 73, 79, 103, 157, 163, 181, 199, 307, 313, 397, 409, 457,\\ 487, 643, 661, 673, 691, 823, 829, 997, 1063, 1447, 1621, 2377, 2689 \}. $
	\item $T(C_4)= \{ 7, 13, 61, 79, 97 \}.$
	\end{itemize} \qed
\end{proof}

\section{Examples}\label{examples section}

At the time of writing there are 161343 rational Bianchi newforms of weight 2 in the LMFDB  \cite{lmfdb} and these are for the quadratic fields $\Q(\sqrt{-d})$ with $d=1,2,3,7,11$. Up to conjugation and twist there are only four genuine newforms for which no corresponding elliptic curve has been found. These are all accounted for by Theorem \ref{examples_thm}.

\begin{curve}\label{C1}
	Let $C_1$ be the genus 2 curve as in Theorem \ref{examples_thm}.1:
	\begin{align*} C_1: y^2= x^6 + 4ix^5 + (-2i - 6)x^4 + (-i + 7)x^3 + (8i - 9)x^2 - 10ix + 4i + 3.
	\end{align*}
	\begin{itemize}
		\item The surface $A=Jac(C_1)$ has conductor $\p_{5,1}^4 \cdot \p_{37,2}^4$ with norm $34225^2$.
		\item  $\O \hookrightarrow \text{End}_{\Q(i)}(A)$ where $\O$ is the maximal order of the rational quaternion algebra of discriminant 6.
		\item There is a genuine Bianchi newform $f \in S_2(\Gamma_0(\p_{5,1}^2 \cdot \p_{37,2}^2))$ which is modular to $A$ and is listed on the LMFDB database with label
		\href{http://www.lmfdb.org/ModularForm/GL2/ImaginaryQuadratic/2.0.4.1/34225.3/a/}{2.0.4.1-34225.3-a.}
	\end{itemize}

\end{curve}

\begin{curve}\label{C2}
	Let $C_2$ be the genus 2 curve as in Theorem \ref{examples_thm}.2:
	\begin{align*}C_2 : y^2=&x^6 + (-2\sqrt{-3} - 10)x^5 + (10\sqrt{-3} + 30)x^4 + (-8\sqrt{-3} - 32)x^3 \\
	&+ (-4\sqrt{-3} + 16)x^2 +     (-16\sqrt{-3} - 12)x - 4\sqrt{-3} + 16.
	\end{align*}
	\begin{itemize}
		\item The surface $A=Jac(C_2)$ has conductor $\p_{13,1}^4 \cdot \p_{19,1}^4$ with norm $61009^2$.
		\item  $\O \hookrightarrow \text{End}_{\Q(\sqrt{-3})}(A)$ where $\O$ is the maximal order of the rational quaternion algebra of discriminant 10.
		\item There is a genuine Bianchi newform $f \in S_2(\Gamma_0(\p_{13,1}^2 \cdot \p_{19,1}^2))$ which is modular to $A$ and is listed on the LMFDB database with label
		\href{http://lmfdb.org/ModularForm/GL2/ImaginaryQuadratic/2.0.3.1/61009.1/a/}{2.0.3.1-61009.1-a.}
	\end{itemize}

\end{curve}

\begin{curve}\label{C3}
	Let $C_3$ be the genus 2 curve as in Theorem \ref{examples_thm}.3:
	\begin{align*}C_3: y^2=&(104\sqrt{-3} - 75)x^6 + (528\sqrt{-3} + 456)x^4 + (500\sqrt{-3} + 1044)x^3\\ &+ (-1038\sqrt{-3} + 2706)x^2
	+ (-1158\sqrt{-3} + 342)x - 612\sqrt{-3} - 1800.
	\end{align*}
	\begin{itemize}
		\item The surface $A=Jac(C_3)$ has conductor $\p_{7,1}^4 \cdot \p_{37,2}^4$ with norm $67081^2$.
		\item  $\O \hookrightarrow \text{End}_{\Q(\sqrt{-3})}(A)$ where $\O$ is the maximal order of the rational quaternion algebra of discriminant 10.
		\item There is a genuine Bianchi newform $f \in S_2(\Gamma_0(\p_{7,1}^2 \cdot \p_{37,2}^2))$  which is modular to $A$ and is listed on the LMFDB database with label
		\href{http://www.lmfdb.org/ModularForm/GL2/ImaginaryQuadratic/2.0.3.1/67081.3/a/}{2.0.3.1-67081.3-a.} 
	\end{itemize}

\end{curve}

\begin{curve}\label{C4}
	Let $C_4$ be the genus 2 curve as in Theorem \ref{examples_thm}.4:
	\begin{align*}C_4 : y^2 = &x^6 - 2\sqrt{-3}x^5 + (2\sqrt{-3} - 3)x^4 + 1/3(-2\sqrt{-3} + 54)x^3\\
	&+ (-20\sqrt{-3} + 3)x^2 + (-8\sqrt{-3} - 30)x + 4\sqrt{-3} - 11.
	\end{align*}
	\begin{itemize}
		\item The surface $A=Jac(C_4)$ has conductor $\p_{3}^{12} \cdot \p_{13,1}^4$ with norm $123201^2$.
		\item  $\O \hookrightarrow \text{End}_{\Q(\sqrt{-3})}(A)$ where $\O$ is the maximal order of the rational quaternion algebra of discriminant 6.
		\item There is a genuine Bianchi newform $f \in S_2(\Gamma_0(\p_{3}^6 \cdot \p_{13,1}^2))$ which is modular to $A$ and is listed on the LMFDB database with label
		\href{http://www.lmfdb.org/ModularForm/GL2/ImaginaryQuadratic/2.0.3.1/123201.1/b/}{2.0.3.1-123201.1-b.}
	\end{itemize}
\end{curve}

\textbf{Acknowledgements.}
 Many thanks to the referees for making useful suggestions and helping to improve the article. I am very grateful to both John Cremona and Aurel Page for kindly computing Bianchi newforms for me. I would like to thank Lassina Demb\'{e}l\'{e} and Jeroen Sijsling for sharing their code with me and theoretical discussions which have been of great help. Finally, it is a pleasure to thank my supervisor Haluk \c{S}eng\"{u}n for his suggestion of this interesting topic and great enthusiasm.

\bibliography{D:/my_work/TeXworks/References}

\end{document}